\pdfoutput=1

\documentclass[10pt]{amsart}

\usepackage{amsmath}

\usepackage{amssymb}
\usepackage{graphicx}

\usepackage{color}
\usepackage[dvipsnames]{xcolor}

\newcommand{\alertm}[1]{%
  \marginpar{%
    \ifodd\value{page} \raggedright \else \raggedleft \fi
    \footnotesize{\textcolor{Green}{#1}}
  }
}

\definecolor{red}{rgb}{1,0,0}

\newtheorem{thm}{Theorem}[section]

\newtheorem{thmx}{Theorem} %non contato
\newtheoremstyle{lemma}{4mm}{1mm}{\itshape}{ }{\bfseries}{.}{ }{}
\newtheorem{prop}[thm]{Proposition}

\theoremstyle{definition}
\newtheorem{defn}[thm]{Definition}

\theoremstyle{remark}
\newtheorem{remark}[thm]{Remark}

\numberwithin{equation}{section}
\newcommand{\mathsym}[1]{{}}
\newcommand{\unicode}[1]{{}}

\def\R{\mathbb{R}}

\def\G{\mathfrak{G}}

\def\t{\mathfrak{t}}

\def\GL{\mathrm{GL}}
\def\O{\mathrm{O}}

\def\T{{^t\!}}
\def\k{\kappa}

\def\PP2{\mathcal{P}^2}
\def\G{\mathrm{G}}
\def\G1{\mathrm{G}^+_1}
\def\NR1{\mathrm{N}^+_1}
\begin{document}

%\tableofcontents
%\newpage

%\title[]{On transversal curves in Lie sphere geometry}
%\title[]{On a variational problem for transversal \\curves in Lie sphere geometry}
\title[]{On a variational problem for curves\\ in Lie sphere geometry}
%\title[]{Transversal curves in Lie sphere geometry}

\author{Lorenzo Nicolodi}
\address{(L. Nicolodi) Dipartimento di Scienze Matematiche, Fisiche e Informatiche, Universit\`a di Parma,
 Parco Area delle Scienze 53/A, I-43124 Parma, Italy}
\email{lorenzo.nicolodi@unipr.it}


%\date{Version of \today}
%\date{Version of June 20, 2023}
%\date{Version of November 18, 2024}
\dedicatory{In memory of Professor Lieven Vanhecke (1939-2023)}


\begin{abstract}

Let $\Lambda$ be
%$T_1(S^3)$ 
the unit tangent bundle of the unit 3-sphere acted on transitively 
%by the group of contact transformations given 
by the contact group of Lie sphere transformations.
We study the Lie sphere geometry of generic curves 
in $\Lambda$ which are 
everywhere transversal to the contact distribution of $\Lambda$.
By the method of moving frames, we prove that such curves
can be parametrized by a 
%natural
Lie-invariant parameter,
the Lie arclength,
and that in this parametrization they are
uniquely determined, up to Lie sphere transformation,
by four local 
%Lie sphere geometry 
invariants, the Lie 
%sphere 
curvatures. 
We then consider the simplest Lie-invariant functional
%variational problem 
on generic transversal curves defined by integrating the differential of the Lie arclength.
The corresponding 
Euler--Lagrange equations are computed and
the critical curves are characterized in terms of their Lie curvatures.
In our discussion, we adopt Griffiths' exterior differential systems 
approach to the calculus of variations.

\end{abstract}

\maketitle

\section{Introduction}\label{s:Intro}

The unit tangent bundle $\Lambda$ 
%$T_1(S^3)$ 
of the unit 3-sphere $S^3$ is acted on transitively 
by the group of Lie sphere transformations, the \textit{Lie sphere group},
%by the group of contact transformations given 
a 
%particular tractable 
subgroup of the group
%the contact transformation group 
of all contact transformations 
of $\Lambda$.
The action 
of the Lie sphere group 
can be linearized by
identifying $\Lambda$ with the Grassmannian of isotropic 2-spaces in %$\mathbb{R}^{4,2}$, that is, 
%$\mathbb{R}^6$
pseudo-Euclidean space $\mathbb{R}^{4,2}$
%with a scalar product
%symmetric bilinear form 
%$h$ 
of signature $(4,2)$.
In this model, the pseudo-orthogonal group $\mathrm{O}(4,2)$ of $\mathbb{R}^{4,2}$ acts on 
$\Lambda$ by contact diffeomorphisms and provides a double cover of the Lie sphere group.
%group of Lie sphere transformations,
The Lie sphere group 
%which naturally 
acts in particular on the 
%space of 
curves of $\Lambda$ which are transversal to the contact distribution of $\Lambda$. 
This paper will investigate the invariant properties of
%Lie sphere invariant properties
transversal curves 
under this action
%in $\Lambda$ that remain invariant under the action of the Lie sphere group
%
and then discuss the critical points of the simplest Lie-invariant variational problem on such curves.
Motivations are provided by
the diffused interest in 
%invariant 
variational problems for curves in homogeneous spaces,
and more specifically in variational problems that are invariant under the 
Lie group
%symmetry group 
%Lie group of symmetries 
of the homogeneous space in question
\cite{EMN-JMAA,FGL,Gr,MN-CQG,MN-CAG},
by optimal control theory \cite{BGar,GM,MN-SIAM},
and by recent work on 
%the geometry of 
transversal curves in the CR 3-sphere \cite{CI,MNS-Kharkiv,MN-olver, MS}.
%%%%
%
In our discussion, we adopt Griffiths' approach to the calculus of variations. This approach,
introduced in 
the book \textit{Exterior differential systems and the calculus of variations}
\cite{Gr}, uses the theory of exterior differential systems, coupled with 
the method of moving frames, as a natural setting for studying variational problems
arising from geometry.
Other geometric applications of Griffiths' formalism can be found in
\cite{BG-ajm,DMN,GM,Hsu,JMN-PhysA,M1}.

%%%%%%%
\vskip0.1cm
Lie sphere geometry was introduced by Sophus Lie around 1870 \cite{Lie, Lie-Scheffers}
%and investigated by Fubini and Cartan at the beginning of the twentieth century.
and developed to a large extent during the 1920s, mainly by W. Blaschke and G. Thomsen \cite{Blaschke}.
Over the past decades, the subject of
Lie sphere geometry has received much attention especially in relation with its applications
to the study of Dupin submanifolds \cite{CC,CC2,Cec-Jen-98, Cec-Jen-00,Miyaoka-1989}, 
Lie applicable surfaces \cite{BH-JPR,MN-tohoku,Pember}, 
and the theory of integrable systems \cite{BH-J,Ferapontov,MN-jmp}. 
The intrinsic structure of Lie sphere geometry was studied
in the late 1980s and early 1990s
%at the end of the 1980s
by Tanaka's theory of graded Lie algebras and Cartan connections \cite{Miyaoka-1991, SY-1989, SY-1993}.
%Recently, 
Interestingly,
Lie sphere geometry has been recently applied to the study of scattering processes in nuclear physics \cite{Ulrych}.

\vskip0.1cm
The first purpose of this paper is to study 
the Lie sphere geometry 
of transversal curves in $\Lambda$.
In our discussion, $\Lambda$ is thought of as a homogeneous space of the identity component $G$ 
of $\mathrm{O}(4,2)$.
By the method of moving frames, for a \textit{generic}\footnote{To be specified in Section \ref{s:MF-for-TC}.} 
transversal curve $\gamma: J \to \Lambda$, 
defined an open interval $J$ of $\R$,
we will construct a canonical lift
$\mathfrak{A}: J\to [G]$ to the quotient Lie group $[G] = G/\{\pm I_6\}$, called the \textit{canonical Lie frame field}.
%$[G] = G/\mathbb{Z}_2$, where $\mathbb{Z}_2=\{\pm I_6\}$.
This allows us to introduce five real-valued invariant functions $s$, $\k_1$, $\k_2$, $\k_3$, $\k_4$, where $s$ is strictly increasing
and $\k_1$ never vanishes. The function $s$ provides a natural parametrization, uniquely defined up to a constant
of integration, which is invariant under the action of $G$. We call $s$ the \textit{Lie arclength} 
and the functions $\k_1$, $\k_2$, $\k_3$, $\k_4$ the \textit{Lie curvatures} of $\gamma$. 
Actually, these five functions %$s, \k_1,\k_2,\k_3,\k_4$
provide a complete set of Lie sphere invariants for generic transversal curves.
More precisely, we prove the following existence and uniqueness result. 

\begin{thmx}\label{thmA}
Given a strictly increasing smooth function $s: J \to \mathbb{R}$ and four smooth functions
$\k_1$, $\k_2$, $\k_3$, $\k_4:J \to \mathbb{R}$, with $\k_1(t)\neq 0$, for each $t\in J$, there exists a generic transversal curve
$\gamma: J \to \Lambda$ 
for which $s$ is the Lie arclength and $\k_1$, $\k_2$, $\k_3$, $\k_4$
are the Lie curvatures. 
In addition,
the curve $\gamma$ 
%is uniquely determined up to Lie sphere transformation,
is uniquely determined up to multiplication on the left by
an element of the Lie group $G$.
%In addition, if $\tilde{\gamma}$
%unique up to Lie sphere transformation, whose Lie sphere curvatures are the
%functions $\k_1$, $\k_2$, $\k_3$, $\k_4$.
\end{thmx}

The second purpose of this paper is the study 
of the Lie-invariant functional 
\begin{equation}\label{inv-funct}
 \mathcal{L} [\gamma] = \int_\gamma ds,
 \end{equation}
%invariant under the Lie sphere group
defined on generic transversal curves by integrating the differential $\eta = ds$ of the Lie arclength.
%on generic transversal curves,
%\[
% \mathcal{W} [\gamma] = \int_\gamma ds.
% \]
%defined on generic transversal curves.
Using Griffiths' exterior differential systems 
approach to the calculus of variations \cite{Gr} (see also \cite{Bryant1987,GM, Hsu}), 
the variational problem defined by \eqref{inv-funct} is replaced by 
%an optimal control problem
%the study of 
a $G$-invariant variational problem for
%whose domain of definition consists of 
the integral curves of a linear control system on $M = [G]\times \R^4(\k_1,\k_2,\k_3,\k_4)$ 
defined by a suitable Pfaffian system $(\mathcal{A},\eta)$
with independence condition $\eta$ (see \eqref{Pfaff-system} and Remark \ref{r:derived-flag}).
For this, we exploit 
the existence of the canonical frame field along generic transversal curves.
We then carry out a general construction due to Griffiths 
%We then follow a general construction due to Griffiths and carry out a calculation 
to associate to the variational problem
a Pfaffian system on a certain affine subbundle $\mathcal{Z}$ of $T^*M$ 
%the phase space. 
{whose integral curves are stationary for the associated functional.
Actually, since all the derived systems of $(\mathcal{A},\eta)$ have constant 
rank, following Bryant \cite{Bryant1987} (see also \cite{Hsu}), we can conclude that all critical curves 
of \eqref{inv-funct} arise as projections of the integral curves of such a Pfaffian system.}
Eventually, we show that generic transversal curves are critical if and only if their Lie curvatures are constant
and satisfy certain additional algebraic conditions (see \eqref{conds-intro} below), which may be
interpreted as the natural equations of the critical curves.
%algebraic relations
%are related by specific equations.
%We will characterize the critical curves of the functional $\mathcal{W}$ in terms of their Lie sphere curvatures.
As a consequence, the critical curves arise as orbits of 1-parameter subgroups of $G$ through
the chosen origin of $\Lambda$.
More precisely, we will prove the following.

\vskip0.3cm
\noindent\textbf{Theorem B}.
%\begin{thmx}\label{thmA}
{\em Let $\gamma: J \to \Lambda$ be a generic transversal curve parametrized by Lie arclength $s$. 
Then, $\gamma$ is a critical curve of the functional $\mathcal{L}$
%an extremal trajectory
if and only if the Lie curvatures $\k_1$, $\k_2$, $\k_3$, $\k_4$ are {constant} and satisfy
the conditions
\begin{equation}\label{conds-intro} 
\k_1  \neq 0, 
%\quad \k_2 = \mathrm{constant},
\quad \k_3 =0, \quad \k_4 = (\k_1)^2-(\k_2)^2.
       \end{equation}
%
%\begin{equation}\label{conds-intro} 
%\k_1= \mathrm{constant} \neq 0, \quad \k_2 = \mathrm{constant},\quad \k_3 =0, \quad \k_4 = (\k_1)^2-(\k_2)^2.
%       \end{equation}
%
Moreover, if $\gamma$ is a critical curve, there exist constants
$u, v\in \R$, $u \neq 0$, such that
\[
 \gamma(s) = A \exp{\left[X(u,v) s\right]} \cdot \lambda_0
  \]
where $s$ is the Lie arclength, $\lambda_0=[e_0,e_1]$ is the origin of $\Lambda$,\footnote{See \eqref{proj}
in Section \ref{s:Lie-geom}.}
$A$ is a fixed element of {$G$}, and
%\[
%X(u,v)= \begin{pmatrix}
%u & -v  & 0    & 1 & u^2-v^2 &  0\\
%v & -u  & 1 & 0   &      0              &  -u^2+ v^2\\
%0              & 0                & 0    & 0   &      1              &  0\\
%0              &  0              &0  & 0   &      0              &  1\\
%1              &  0              & 0  & 0   &      u               & v\\
%0              &  -1              & 0  & 0   &      -v             &  -u\\
% \end{pmatrix} \in \mathfrak{g}
% \]
 \begin{equation}\label{X(a1,a2)-intro}
 X(u,v)= 
 \begin{pmatrix} u I_{1,1} + v LI_{1,1} & L & (u^2-v^2)I_{1,1} \\ 
 0 & 0 & I_2\\ 
 I_{1,1} & 0 &  u I_{1,1} - v LI_{1,1} \end{pmatrix} \in \mathfrak{g},
 \end{equation}
 where $\mathfrak{g}= \mathrm{Lie}(G)$, $L=(\begin{smallmatrix} 0&1\\1&0\end{smallmatrix})$, $I_{1,1}=(\begin{smallmatrix} 1&0\\0&-1\end{smallmatrix})$, 
 and $I_2$ is the $2\times 2$ identity matrix.}
%\end{thmx}
\vskip0.3cm

The paper if organized as follows. 
Section \ref{s:Lie-geom} briefly recalls the basic structure of Lie sphere geometry.
Section \ref{s:MF-for-TC} develops the method of moving Lie frames for transversal curves in $\Lambda$.  
We construct the canonical Lie frame field for a generic transversal curve and introduce its
Lie sphere invariants, namely the Lie arclength $s$ and the Lie sphere 
curvatures $\kappa_1, \kappa_2,\kappa_3, \kappa_4$. The results 
of this section are summarized in Proposition \ref{Canonical-frame}. 
Theorem \ref{thmA} is then proved as a consequence of Proposition \ref{Canonical-frame} 
and the Cartan--Darboux existence and uniqueness theorems for smooth
maps of a manifold into a Lie group. 
Finally, Section \ref{s:var-pbm} gives a proof of Theorem \ref{thmB} using the Griffiths formalism for the calculus of variations
in one variable.

\vskip0.1cm
\noindent\textit{Acknowledgements}. 
The author gratefully acknowledges valuable input by Emilio Musso.
%The author would like to thank Emilio Musso for valuable input. 
This work was partially supported by PRIN 2017 and 2022 ``Real and Complex 
Manifolds: Topology, Geometry and Holomorphic Dynamics" (protocolli 2017JZ2SW5-004 and 
2022AP8HZ9-003); and by GNSAGA of INdAM.

\section{Lie sphere geometry}\label{s:Lie-geom}

%\subsection{The Lie sphere group}\label{ss:Lie-geom}

We begin by recalling the basic structure of Lie sphere geometry.
More details about Lie sphere geometry are given in the monograph of Cecil \cite{Ce}, in the
book of Blaschke and Thomsen \cite{Blaschke}, in Lie's original work \cite{Lie,Lie-Scheffers}, or in \cite{JMNbook}.

Let $\R^{4,2}$ denote $\R^6$ with the orientation induced by the standard basis
$e_0,\dots,e_5$ and with
the symmetric bilinear form
\begin{equation}\label{lie-sp}
\langle x, y \rangle=-(x^0y^5 + x^5y^0)-(x^1y^4 + x^4y^1)  +x^2y^2 +x^3y^3 = \T x h\, y
\end{equation}
of signature $(4,2)$,
where
\begin{equation}\label{h-L}
h = (h_{ij}) = 
\begin{pmatrix} 0 & 0 & -L\\
0& I_2 & 0\\
-L & 0 & 0\end{pmatrix}, 
\quad L =  \begin{pmatrix} 0& 1\\1 & 0\end{pmatrix},
\quad I_2=  \begin{pmatrix} 1& 0\\0 & 1\end{pmatrix}.
\end{equation}
Let $G$ be the identity component of the pseudo-orthogonal group of \eqref{lie-sp},
%\textcolor{red}{Lie sphere group}
\[
             \{A \in \GL(6,\R) \mid \T A h A =h\} \cong \O(4,2),
     \]
%The group $G$ is called the \textit{Lie sphere group}.
and let
\[
 \mathfrak g = \{ X  \in \mathfrak{gl}(6,\R) \mid \T Xh + hX = 0  \}
    \]
be the Lie algebra of $G$. Following \cite{Cec-Jen-98}, it is useful to understand the elements of $\mathfrak{g}$ in the block form
\begin{equation}\label{g-block-form}
 \mathfrak g = \left\{ \begin{pmatrix} X_1 & \T X_2 & tI_{1,1} \\ X_4 & X_5 & X_2 L\\ s I_{1,1} & L \T X_4 & -L\T X_1 L\end{pmatrix}  
% \mid 
%\,\, \Bigg\vert\,\,
\left\vert \,\,
% \begin{matrix} X_1 \in \mathfrak{gl}(2,\R), & s, t \in\R \\
%X_5 \in \mathfrak{so}(2,\R), & X_2, X_4 \in \mathfrak{gl}(2,\R)\end{matrix}  \right\},
 \begin{matrix} X_1, X_2, X_4 \in \mathfrak{gl}(2,\R) \\ 
X_5 \in \mathfrak{so}(2,\R)\\ s, t \in\R \end{matrix} 
\right. 
\right\},
    \end{equation}
where $I_{1,1} = \begin{pmatrix} 1&0\\0&-1 \end{pmatrix}$.
For every $A\in G$, let $A_j= Ae_j$ denote the $j$th column vector of $A$.
Regarding each of the vectors $A_j$ as an $\R^6$-valued function on $G$, since the $A_j$ form a basis, there
exist unique 1-forms $\omega^i_j$, $i,j \in \{0,\dots,5\}$, so that
\begin{equation}\label{dAj}
dA_j = \sum_i \omega^i_j A_i, \quad j = 0,\dots,5.
\end{equation}
The 1-forms $\omega^i_j$ are the components of the Maurer--Cartan form $\omega =A^{-1}dA$ of $G$.
Using the block form $\eqref{g-block-form}$ for $\mathfrak{g}$, the Maurer--Cartan form $\omega$ can be written as
\begin{equation}\label{omega-block-form}
\omega = (\omega^i_j) = \begin{pmatrix} \Omega_1 & \T \Omega_2 & \omega^0_4 I_{1,1} \\ \Omega_4 & \Omega_5 & \Omega_2 L
\\ \omega^0_4 I_{1,1} & L \T \Omega_4 & -L\T \Omega_1 L\end{pmatrix} ,
%\mid 
% \begin{matrix} X_1 \in \mathfrak{gl}(2,\R) & s, t \in\R \\
%X_5 \in \mathfrak{so}(2,\R) & X_2, X_4 \in \mathfrak{gl}(2,\R)\end{matrix}
\end{equation}
where the blocks are left-invariant 1-forms satisfying the following relations:
%taking values as follows:
\begin{equation}\label{blocks-omega}
\begin{array}{ll}
 \begin{pmatrix} \omega^0_0 & \omega^0_1\\\omega^1_0& \omega^1_1\end{pmatrix} =\Omega_1, &
 \begin{pmatrix} \omega^4_4 & \omega^4_5\\\omega^5_4& \omega^5_5\end{pmatrix} = -L\T \Omega_1 L =
\begin{pmatrix} -\omega^1_1 & -\omega^0_1\\ -\omega^1_0& -\omega^0_0\end{pmatrix}\\
 \begin{pmatrix} \omega^0_2 & \omega^0_3\\\omega^1_2& \omega^1_3\end{pmatrix} = \T\Omega_2 &
 \begin{pmatrix} \omega^2_4 & \omega^2_5\\\omega^3_4 & \omega^3_5\end{pmatrix} =\Omega_2 L =
\begin{pmatrix} \omega^1_2 & \omega^0_2\\\omega^1_3 &\omega^0_3\end{pmatrix}, \\
  \begin{pmatrix} \omega^2_0 & \omega^2_1\\\omega^3_0 & \omega^3_1\end{pmatrix}= \Omega_4, &
 \begin{pmatrix} \omega^4_2 &\omega^4_3\\\omega^5_2& \omega^5_3\end{pmatrix} =L  \T \Omega_4 =
\begin{pmatrix} \omega^2_1 & \omega^3_1\\\omega^2_0& \omega^3_0\end{pmatrix}, \\
\begin{pmatrix} 0 & \omega^2_3\\ \omega^3_2& 0\end{pmatrix} =\Omega_5 = - \T \Omega_5,\\
\end{array}
\end{equation}
\begin{equation}
\begin{matrix}\label{bloks-omega1}
\omega^4_0 = -\omega^5_1, & \omega^4_1 = 0 = \omega^5_0, &\omega^0_4 = -\omega^1_5, &
\omega^0_5 = 0 = \omega^1_4.
\end{matrix}
\end{equation}

The 1-forms $\omega^0_0$, $\omega^0_1$, $\omega^1_0$, $\omega^1_1$, $\omega^2_0$, $\omega^2_1$,
$\omega^3_0$, $\omega^3_1$, $\omega^4_0$, $\omega^3_2$, $\omega^0_4$, $\omega^2_4$, $\omega^2_5$, $\omega^3_4$,
$\omega^3_5$, yield a left-invariant coframe field on $G$. Differentiating \eqref{dAj}, we obtain the 
Maurer--Cartan structure equations
\begin{equation}\label{struct-eqs}
 d\omega^i_j = - \sum_{k=0}^5 \omega^i_k \wedge \omega^k_j, \quad i,j \in \{0, \dots,5\}.
  \end{equation}

%%%%%%%

Let $\mathbb{RP}^5$ be the 5-dimensional real projective space. The \textit{Lie quadric} is 
\[
 \mathcal Q = \left\{[x] \in \mathbb{RP}^5 \mid \langle x, x\rangle = 0  \right\}.
  \]
In Lie sphere geometry, the Lie quadric parametrizes the set
of all oriented 2-sphere in $S^3$, including points, and the lines in $\mathcal Q$ correspond to parabolic
pencils of spheres in oriented contact. Let $[x,y]$ denote the line joining the points $[x]$ and $[y]$ in $\mathcal Q$.
It is known that the set $\Lambda$ of all lines in $\mathcal Q$, that is, the 
Grassmannian of null 2-planes through the origin in $\R^{4,2}$, forms a 5-dimensional
smooth manifold which can be identified with the unit tangent bundle $T_1(S^3)$ of $S^3\subset \R^4$.
{Under this identification}, by the standard action of $G$ on $\mathbb{RP}^5$, the quadric $\mathcal Q$ is mapped into itself.
Moreover, this action takes lines to lines and induces a transitive action on $\Lambda$.

\begin{remark}\label{r:T1S3}
{Viewing $T_1(S^3)$ as the set of all pairs $(v,\xi)\in S^3 \times S^3$ with $v\cdot \xi=0$, where $\cdot$
denotes the scalar product of $\R^4$,
 \[
 T_1(S^3) = \left\{(v, \xi) \in S^3\times S^3\subset \R^4 \times \R^4  \mid  v\cdot \xi =0 \right\},
 \]
 the identification of $T_1(S^3)$ with the Grassmannian $\Lambda$ can be realized by the map
 \[
  T_1(S^3) \ni (v,\xi) \longmapsto [A_0(v), A_1(\xi)]\in \Lambda,
    \]
    where, for $v= \T(v^0,v^1,v^2,v^3)$ and $\xi=\T(\xi^0,\xi^1,\xi^2,\xi^3)$,
    \begin{eqnarray*}
    A_0(v) &=& \frac12(1-v^0) e_0 - \frac12 v^1 (e_1-e_4) +\frac{1}{\sqrt2}v^2e_2 +\frac{1}{\sqrt2}v^3e_3 +\frac12(1+v^0)e_5,\\
    A_1(\xi) &=& -\frac12 \xi^0 (e_0-e_5) + \frac12 (1- \xi^1) e_1 +\frac{1}{\sqrt2}\xi^2e_2 +\frac{1}{\sqrt2}\xi^3e_3 +\frac12(1+\xi^1)e_4.
    \end{eqnarray*}
    }
%(see \cite{Ce, JMNbook} for more details).
\end{remark}

\vskip0.1cm
Let $\lambda_0=[e_0,e_1]$ be the chosen origin for $\Lambda$. The projection map
\begin{equation}\label{proj}
  \pi_\Lambda : G \ni A \longmapsto A[e_0,e_1] =[Ae_0, Ae_1]= [A_0, A_1] \in \Lambda = G/G_0
   \end{equation}
is then a principal $G_0$-bundle, where $G_0$ is the isotropy subgroup of $G$ at $\lambda_0$.
%$[e_0,e_1]$.
The isotropy subgroup $G_0$ is
\[
G_0 = 
\left\{ \begin{pmatrix} L\T C^{-1}L & \T Z & b \\ 0 & B & BZLC\\ 0 & 0 & C\end{pmatrix}  
%\mid 
\left\vert \,\,
 \begin{array}{ll} C \in \mathrm{GL}_+(2,\R), & b \in \mathfrak{gl}(2,\R) \\
B \in \mathrm{SO}(2,\R),  & Z \in \mathfrak{gl}(2,\R)\\
bC^{-1}L + L \T C^{-1} \T b = \T Z Z\end{array}  
\right.
\right\}.
\]
The elements of $G_0$ will be denoted by
\[
a(C,B,Z,b) = \begin{pmatrix} L\T C^{-1}L & \T Z & b \\ 0 & B & BZLC\\ 0 & 0 & C\end{pmatrix}.
\]
The inverse of $a(C,B,Z,b)$ is given by
\[
a(C,B,Z,b)^{-1} = \begin{pmatrix} L\T C L & -L \T C L \T Z \T B & L\T b L \\ 0 & \T B & -ZL\\ 0 & 0 & C^{-1}\end{pmatrix}.
\]

\begin{defn}
A \textit{Lie frame field} in $\Lambda$ is a local smooth section $A$ of \eqref{proj}, that is,
a smooth map $A : U \to G$, from an open set $U\subset \Lambda$, such that $\pi_\Lambda \circ A = [A_0,A_1]$.
\end{defn}

If $A : U \to G$  is a Lie frame field, writing $\omega^i_j$ for the entries $A^*(\omega^i_j)$ of $A^*(\omega)$,
which are now 1-forms on $U$,
it follows that the 1-forms 
\begin{equation}\label{coframe-on-U}
 \omega^4_0, \, \omega^2_0, \, \omega^2_1, \,  \omega^3_0, \,  \omega^3_1
  \end{equation}
define a coframe field on $U\subset \Lambda$.

Any other Lie frame field on $U$ is given by 
\[
  \hat{A} = A a(C,B,Z,b),
   \]
where $a : U \to G_0$ is an arbitrary smooth map. This implies
\begin{equation}\label{transf-rule}
\hat{\omega}= a^{-1} \omega a + a^{-1} da.
   \end{equation}
From this and \eqref{omega-block-form}, it follows that
\begin{equation}\label{transf-omega^4_0}
\hat{\omega}^4_0 I_{1,1} = \det (C^{-1}) \omega^4_0 I_{1,1} \iff \hat{\omega}^4_0 = \frac{1}{\det C} \omega^4_0,
\end{equation}
\begin{equation}\label{transf-Omega_4}
 \begin{pmatrix} \hat{\omega}^2_0& \hat{\omega}^2_1\\\hat{\omega}^3_0&\hat{\omega}^3_1 \end{pmatrix} = 
 \widehat{\Omega}_4  =
\T B  \Omega_4 L \T C^{-1} L + \omega^4_0 Z I_{1,1} \T C^{-1} L.
\end{equation}

%\subsubsection{The contact structure on $\Lambda$}
From the structure equations \eqref{struct-eqs}, using the relations \eqref{blocks-omega}, we have
\begin{equation}\label{domega40}
 d\omega^4_0 \equiv  
   -\omega^2_1 \wedge \omega^2_0 - \omega^3_1 \wedge \omega^3_0  
     \mod \omega^4_0.
      \end{equation}
Since $\omega^4_0$, $\omega^2_0$, $\omega^2_1$, $\omega^3_0$, $\omega^3_1$ gives a coframe on $U$, 
it follows from \eqref{domega40} that 
\[
 \omega^4_0 \wedge d\omega^4_0 \wedge d\omega^4_0 \neq 0.
  \]
Hence, according to \eqref{transf-omega^4_0}, 
as $A:U \to G$ varies through all Lie frame fields in $\Lambda$, 
the set of 1-forms $A^*(\omega^4_0) {= - \langle dA_0 , A_1\rangle}$ 
defines a contact structure on $\Lambda$.

\section{Moving Lie frames for transversal curves in $\Lambda$}\label{s:MF-for-TC}

In this section, we develop the method of moving Lie frames for the class of transversal curves in $\Lambda$.
By a \textit{transversal curve}, or \textit{T-curve}, %$\gamma : J \to \Lambda$ 
we mean a
%we consider 
smooth immersed parametrized curve $\gamma : J \to \Lambda$, from a connected open
set $J\subset\mathbb R$ to $\Lambda$, which is everywhere
transverse to the contact distribution induced on $\Lambda$ by the 1-form $\omega^4_0$,
%%%
{that is, $\alpha(\gamma'(t)) \neq 0$, for every $t\in J$ and for every local contact form $\alpha=A^*(\omega^4_0)$.}
%(see previous section).

\begin{remark}
{With reference to Remark \ref{r:T1S3}, let $\pi_1, \pi_2 :T_1(S^3) \to S^3$ denote the restrictions to $T_1(S^3)$
of the canonical projections of $S^3\times S^3$ onto its factors. Then, the equation $d\pi_1 \cdot \pi_2=0$
defines a contact distribution on $T_1(S^3)$ which 
under the identification $\Lambda\cong T_1(S^3)$ corresponds to the contact distribution induced on 
$\Lambda$ by 1-form $\omega^4_0$.
Accordingly, a transversal curve $\gamma (t)= (v(t), \xi(t))$ in $T_1(S^3) \cong \Lambda$ can be viewed
%consists
as an immersed curve $v(t)$ in $S^3$, together with a unit vector field $\xi$ along $v$, whose tangential
component is never zero.
}
\end{remark}

\begin{defn}
A \textit{Lie frame field along} a \textit{T-curve} $\gamma: J \to \Lambda$ is a smooth map
$A : I \subset J \to G$, defined on an open subset $I$ of $J$, such that the projection \eqref{proj} composed 
with $A$ is $\gamma$, that is, $\gamma(t) = [A_0(t), A_1(t)]$, for every $t \in I$.
\end{defn}

If $A : I \to G$ is a Lie frame field along $\gamma$, any other Lie frame field on $I$ is given by
\begin{equation}\label{Lie-frames-relat}
\hat{A} = A a(C,B,Z,b),
\end{equation}
where $a : I \to G_0$ is an arbitrary smooth map. Given a Lie frame field $A$, we denote by $\theta =(\theta^i_j)$
the pullback of the Maurer--Cartan form $\omega$ of $G$ by $A$. With this notation, the $\mathfrak{g}$-valued
1-forms $\theta$ and $\hat{\theta}$ are related by
\begin{equation}\label{trans-rules-1}
 \hat{\theta} = a^{-1}\theta a + a^{-1}da.
  \end{equation}
Since $\gamma$ is everywhere transversal to the contact distribution, from the transformation 
rules \eqref{transf-omega^4_0} 
and \eqref{transf-Omega_4} it follows that, 
%on a neighborhood of 
about every point, there exists a 
Lie frame field along $\gamma$,
such that
\begin{equation}\label{1st-order}
 \theta^4_0 = \rho(t) dt, \, \rho(t)> 0, \quad {\theta^2_0}(t)= {\theta^2_1}(t)= {\theta^3_0}(t) = {\theta^3_1}(t) = 0.
 %, \quad \forall\, t\in J.
\end{equation}

\begin{defn}
A Lie frame field $A:I\to G_0$ is of \textit{first order} if it satisfies \eqref{1st-order} at every point of $I$.
%A \textit{first order} Lie frame field along a transversal curve $\gamma$ is a Lie frame field satisfying \eqref{1st-order}.
\end{defn}

If $A$ is a first order Lie frame field, then any other on $I$ is given by \eqref{Lie-frames-relat}, where $a: I \to G_1$
is a smooth map, and
\[
G_1 = 
\left\{ \begin{pmatrix} L\T C^{-1}L & 0 & b \\ 0 & B & 0\\ 0 & 0 & C\end{pmatrix} 
% \mid 
\left\vert \,\,
 \begin{array}{cc} C \in \mathrm{GL}_+(2,\R), & b \in \mathfrak{gl}(2,\R) \\
B \in \mathrm{SO}(2,\R), & 
bC^{-1}L + L \T C^{-1} \T b = 0\end{array}  
\right.
\right\}.
\]
For a given first order Lie frame field $A:I \to G_1$, let
\[
 \begin{pmatrix} \theta^2_4 & \theta^2_5\\\theta^3_4 & \theta^3_5\end{pmatrix} = P \theta^4_0,
\]
where $P$ is a smooth $2\times 2$ matrix-valued function. If $\hat{A}$ is any other first order Lie frame field
on $I$, it follows from \eqref{trans-rules-1} that
\[
 \begin{pmatrix} \hat{\theta}^2_4 & \hat{\theta}^2_5\\\hat{\theta}^3_4 & \hat{\theta}^3_5\end{pmatrix} 
 = \T B  \begin{pmatrix} \theta^2_4 & \theta^2_5\\\theta^3_4 & \theta^3_5\end{pmatrix}  C,
\]
which implies 
\[
\hat{P} \hat{\theta}^4_0 = \T B PC \theta^4_0.
\]
From this, since $\hat{\theta}^4_0 = \frac{1}{\det (C)} \theta^4_0$ by \eqref{transf-omega^4_0}, we obtain
\begin{equation}\label{transf-of-P}
  \hat{P}  = \det (C)\, \T B PC.
   \end{equation}
The rank of $P(t)$ is independent of 
the choice of first order Lie frame field about $t$.

\begin{defn}
We shall say that the T-curve $\gamma : J \to \Lambda$ is \textit{nondegenerate} if $\det P \neq 0$
at every point of $J$. For a nondegenerate $\gamma$, the first order Lie frame fields may be further specialized
in such a way that $P = I_2$. 
A first order Lie frame field $A:I\to G_1$ is of \textit{second order} if $P=I_2$, 
at every point of $I$.
\end{defn}

If $A$ is a second order Lie frame field, then any other on $I$ is given by \eqref{Lie-frames-relat}, where $a: I \to G_2$
is a smooth map, and
\[
G_2 = 
\left\{ \begin{pmatrix} \T B & 0 & LB \mathfrak{h} \\ 0 & B & 0\\ 0 & 0 & B\end{pmatrix}  \mid 
 \begin{array}{cc} B \in \mathrm{SO}(2,\R), & \mathfrak{h} \in \mathfrak{so}(2,\R) \\
%B \in \mathrm{SO}(2,\R), & bB^{-1}L + L \T B^{-1} \T b = 0
\end{array}  
\right\}.
\]
It follows from \eqref{transf-omega^4_0} that the 1-form $\theta^4_0$ is independent of the choice of second order Lie frame
field. Moreover, since $\theta^4_0 = \rho dt$ and $\rho > 0$, there exists a strictly increasing function $s: J \to \mathbb R$, 
such that $\theta^4_0=ds$. The function $s$ is uniquely determined, up to a constant of integration.

\begin{defn}\label{natural-param}
The strictly increasing function $s$ such that $\theta^4_0=ds$ is called the \textit{Lie arclength} of the nondegenerate transversal 
curve $\gamma$.
\end{defn}

For a given second order Lie frame field $A : I \to G_2$, let
\[
 \Theta_1 = \begin{pmatrix} \theta^0_0 & \theta^0_1\\ \theta^1_0 & \theta^1_1 \end{pmatrix}.
  \]
Let $\hat{A}= A a$ be any other second order Lie frame field on $I$, where $a : I \to G_2$. 
Using the calculations $\T B L = LB$ and $B L I_{1,1} \T B 
= \det (\T B) L I_{1,1} = L I_{1,1}$, it follows from \eqref{trans-rules-1}
that
%the following transformation formulas
\begin{equation}\label{transf-Theta1}
 \widehat{\Theta}_1 = B\left(\Theta_1 -h I_2 \theta^4_0\right) \T B + B d (\T B),
   \end{equation}
where we have set $\mathfrak{h} = \begin{pmatrix} 0 & -h \\ h & 0\end{pmatrix}$. According to these transformation formulas,
%\eqref{transf-Theta1}, 
there exist second order Lie frame fields satisfying the condition
\begin{equation}\label{3rd-order}
 \theta^0_0 + \theta^1_1 = 0.
\end{equation}

\begin{defn}
A second order Lie frame field $A:I\to G_2$ is of \textit{third order} if it satisfies \eqref{3rd-order} at every point of $I$.
\end{defn}

If $A$ is a third order Lie frame field on $I$, then any other third order Lie frame field on $I$ is given by $\hat{A} = A a$, 
where $a: I \to G_3$ is a smooth map, and
\[
G_3 = 
\left\{ \begin{pmatrix} \T B & 0 & 0 \\ 0 & B & 0\\ 0 & 0 & B\end{pmatrix}  \mid 
 \begin{array}{c} B \in \mathrm{SO}(2,\R)
 %& f \in \mathfrak{so}(2,\R) \\
%B \in \mathrm{SO}(2,\R), & bB^{-1}L + L \T B^{-1} \T b = 0
\end{array}  
\right\}.
\]
For a given third order Lie frame field $A: I \to G_3$, we set
%\[
% \frac12(\theta^0_1 + \theta^1_0)=: E\, \theta^4_0, \quad \theta^1_1 =: F\,\theta^4_0,
%  \]
  \[
 \frac12(\theta^0_1 + \theta^1_0)=: p\, \theta^4_0, \quad \theta^1_1 =: q\,\theta^4_0,
  \]
where $p$ and $q$ are smooth functions. Let $\hat{A} = A a$ be any other third order frame field defined on $I$,
where $a: I \to G_3$, 
and write $B \in \mathrm{SO}(2,\mathbb{R})$ as $B = \begin{pmatrix} s&-t\\t & s\end{pmatrix}$.
It follows from \eqref{transf-Theta1} that $\hat{q} = (s^2-t^2) q -2st p$ and $\hat{p} = 2stq + (s^2-t^2) p$, that is,
$\hat{p}$, $\hat{q}$ and $p$, $q$ are related by
\[
  \begin{pmatrix} \hat{q} \\ \hat{p} \end{pmatrix} = B^2 \begin{pmatrix} {q} \\ {p} \end{pmatrix}.
   \]
In particular, we have that $p^2 +q^2 $ is a well defined real valued function.

{Moreover, from the transformation rule \eqref{transf-rule} it also follows that}
\begin{equation}\label{theta04-inv}
  \hat{\theta}^0_4 I_{1,1} = \theta^0_4 B I_{1,1} B = \theta^0_4 I_{1,1}.
   \end{equation}

\begin{defn}
We shall say that a nondegenerate T-curve $\gamma : J \to \Lambda$ is \textit{generic} if $p^2 + q^2 >0$
at every point of $J$.
\end{defn}

For a generic curve, third order frame fields my be further specialized by requiring that $p=0$. 

\begin{defn}
A third order Lie frame field $A:I\to G_3$ is of \textit{fourth order} if $p=0$ at every point of $I$.
\end{defn}

If $A$ is a fourth order Lie frame field on $I$, then any other fourth order Lie frame field on $I$ is given by $\hat{A} = A a$,
where $a: I \to G_4$, and
\[
G_4= 
%\mathbb Z_2 := 
\left\{ \begin{pmatrix} \pm I_2 & 0 & 0 \\ 0 & \pm I_2 & 0\\ 0 & 0 & \pm I_2 \end{pmatrix} \right\} 
%= \{ \pm I_6 \} 
=:\mathbb Z_2.
\]

Let $[G]$ denote the quotient Lie group $[G] = G/\mathbb{Z}_2$ and for $A \in G$ let $[A]$ denote its equivalence class in $[G]$. 
Thus $[A] = [B]$ if and only if $B = \pm A$.
The quotient group $[G] = G/\mathbb{Z}_2$ acts effectively on $\Lambda$. The principal fibration $\pi_\Lambda : G \to \Lambda$
is $\mathbb{Z}_2$-invariant and induces a principal fibration of $[G]$ over $\Lambda$. The fourth order frame fields along 
a generic T-curve $\gamma$ give rise to a well defined 
%cross section 
lift
$\mathfrak{A} = [A]$ to $[G]$,
%of $[G] \to \Lambda$,
where $A$ is any fourth order Lie frame field. 
%and [A] denotes the equivalence class in $[G]$.

We call the lift
%cross section 
$\mathfrak{A}: J \to [G]$ the \textit{canonical frame field} of the generic curve $\gamma : J \to \Lambda$.
Observe that the pullback $\mathfrak{A}^*(\omega)$ coincides with $\theta= A^*(\omega)$, where $A$
is any fourth order Lie frame field along $\gamma$.

Let $\mathfrak{A} : J \to [G]$ be the canonical frame field of a generic T-curve $\gamma : J \to \Lambda$,
parametrized by Lie arclength $s$. 
{The components of $\theta= A^*(\omega)$ not already specified in the previous steps of the reduction,
namely $\theta^0_0 = -\theta^1_1$, $\theta^1_0 = -\theta^0_1$, and $\theta^3_2$,
must be multiples of $ds$, and these multiples are differential invariants.}

The above discussion allows us to set
\[
 \theta^0_0 = \kappa_1 ds, \quad \theta^1_0 = \kappa_2 ds, \quad \theta^3_2 = \kappa_3 ds, \quad \theta^0_4 = \kappa_4 ds,
\]
where $\kappa_1\neq 0$, $\kappa_2$, $\kappa_3$, $\kappa_4: J \to \mathbb R$ are real-valued smooth functions. 
\begin{defn}
We call 
$\kappa_1$, $\kappa_2$, $\kappa_3$, $\kappa_4$ the \textit{Lie curvatures} of $\gamma$ (with respect to the
canonical frame field).
\end{defn}

\begin{remark}
{From the reduction procedure it follows that $\kappa_4$ is a third order invariant (cf. \eqref{theta04-inv}), while
 $\kappa_1, \kappa_2$, and $\kappa_3$ are fourth order invariants.}
% The Lie arclength function $s$ is a second order invariant.
\end{remark}

\vskip0.1cm
Summarizing, we have proved the following.

\begin{prop}\label{Canonical-frame}
Let $\gamma : J \to\Lambda$ be a generic T-curve.
There exists 
%a natural parameter $s$ for $\gamma$ and 
a unique frame field $\mathfrak{A} =[A] : J \to [G]$ along $\gamma$,
the {\em canonical frame} of $\gamma$,
such that 
\begin{equation}\label{Normal-frame}
\mathfrak{A}^*(\omega) =
%\theta =
 A^*(\omega) =
%A^{-1}dA=
%\begin{pmatrix}
%\kappa_1ds & -\kappa_2 ds & 0    & ds & \kappa_4 ds &  0\\
%\kappa_2ds & -\kappa_1 ds & ds  & 0   &      0              &  -\kappa_4 ds\\
%0              & 0                & 0    & -\kappa_3 ds   &      ds              &  0\\
%0              &  0              & \kappa_3ds  & 0   &      0              &  ds\\
%ds              &  0              & 0  & 0   &      \kappa_1 ds              &  \kappa_2ds\\
%0              &  -ds              & 0  & 0   &      -\kappa_2 ds              &  -\kappa_1ds\\
% \end{pmatrix},
% \end{equation}
% \[
  \begin{pmatrix} \k_1I_{1,1} + \k_2LI_{1,1} & L & \k_4I_{1,1} \\ 
  0 & \k_3L I_{1,1} & I_2\\ 
  I_{1,1} & 0 & \k_1I_{1,1} - \k_2LI_{1,1} \end{pmatrix}ds,
  %\]
  \end{equation}
where $s$ is the {\em Lie arclength},  
the functions 
$\kappa_1\neq 0$, $\kappa_2$, $\kappa_3$, $\kappa_4 : J \to \mathbb R$ 
%are smooth functions, 
%$\kappa_1(s) \neq 0$, for every $s\in J$, called the
are the {\em Lie curvatures} of $\gamma$,
%\begin{equation}
%  \kappa =\frac{1}{2}\langle F_1',F_1'\rangle, 
%  \quad   \tau ={\mathrm{Im}}\left(\langle F_1'',F_1'\rangle\right)+3\kappa^2.
%\end{equation}
and $A$ is any fourth order Lie frame field along $\gamma$. 
%Any other is given by $\pm A$. 
%
%Thus, there exists a unique frame
%field $F:= [A] : J \to [G]$ along $\gamma$, called the {\em canonical frame} of $\gamma$, such that $F^*(\omega) = \theta$.
\end{prop}

 \begin{remark}
As a consequence of Proposition \ref{Canonical-frame} and the Cartan--Darboux existence and congruence theorems 
for smooth maps of a manifold into a Lie group (see \cite{Gr,JMNbook}),
given four smooth functions $\kappa_1\neq 0$, $\kappa_2$, $\kappa_3$, $\kappa_4 : J \to \mathbb R$
and a strictly increasing function $s: J\to \R$,
there exists a 
generic T-curve $\gamma:J\to\Lambda$, unique up to Lie sphere transformation,
for which $s$ is
%parametrized by 
the Lie arclength and
%whose Lie sphere curvatures are 
$\kappa_1$, $\kappa_2$, $\kappa_3$, $\kappa_4$ are the Lie curvatures. This proves Theorem \ref{thmA}.
%The curve $\gamma$ is unique up to Lie sphere transformations of $\Lambda$.
\end{remark}

%%%%%%%%%%

\section{A variational problem for transversal curves in $\Lambda$}\label{s:var-pbm}

Let $\mathfrak{T}$ denote the space of generic T-curves in $\Lambda$ which are parametrized by Lie arclength.
Consider the Lie-invariant functional $\mathcal{L} : \mathfrak{T} \to \mathbb R$, defined by
\begin{equation}\label{action-funct}
% \mathcal{W} : \gamma  \in \mathfrak{T} \longmapsto \int_{J_\gamma} { \tau_\gamma\, \eta_\gamma} \,,
  \mathcal{L}[\gamma] =  \int_{J_\gamma} { ds} \,,
   \end{equation}
%defined on the space $\mathfrak{T}$ of generic transversal curves in $\mathcal{S}$,
where ${J_\gamma}$ is the definition domain of $\gamma$
and $s$ is the Lie arclength of $\gamma$ (see Def. \ref{natural-param}).

%is the infinitesimal strain of $\gamma$ (cf. Section \ref{ss:canonical-frame}).

A T-curve $\gamma$ of $\mathfrak{T}$ is a \textit{critical curve}
%an \textit{extremal trajectory} (or simply a \textit{trajectory})
%in $\Lambda$ 
if it is a critical point of 
%the functional 
$\mathcal{L}$ when considering compactly 
supported variations through generic T-curves.

\begin{remark}
As usual, by a compactly supported variation of $\gamma \in \mathfrak T$ is meant a mapping 
$V : J \times (-\epsilon, \epsilon) \to \Lambda$, such that: (1) for all $u\in (-\epsilon, \epsilon)$, the
map $\gamma_u = V(s,u): J \to \Lambda$ is a generic T-curve; (2) $\gamma_0 = \gamma(s)$, for all $s \in J$;
and (3) there exists a closed interval $[a,b] \subset J$ such that
\begin{equation}\label{variation}
  V(s,u) = \gamma(s), \quad \forall\, s \in J\setminus [a,b], \quad \forall\, u \in (-\epsilon, \epsilon).
\end{equation}
Accordingly, $\gamma\in \mathfrak T$ is a critical curve if, for every compactly supported variations $V$,
we have that 
\[
\frac{d}{du}\left.\left( \int_{a_V}^{b_V}ds_u\right)\right|_{u=0} =0,
	\]
where $[a_V,b_V]$ is the support of the variation, that is, the smallest closed interval for which 
\eqref{variation} holds, and $s_u$ is the Lie arclength of the T-curve $\gamma_u$.
\end{remark}

\vskip0.1cm
In this section we will prove the following.

\begin{thmx}\label{thmB}
Let $\gamma: J \to \Lambda$ be a generic T-curve parametrized by Lie arclength $s$. 
Then, $\gamma$ is a critical curve of the functional $\mathcal{L}$
%an extremal trajectory
if and only if the Lie curvatures $\k_1$, $\k_2$, $\k_3$, $\k_4$ are {constant} and satisfy
the conditions
\begin{equation}\label{conds} 
\k_1 \neq 0, 
%\quad \k_2 = \mathrm{constant},
\quad \k_3 =0, \quad \k_4 = (\k_1)^2-(\k_2)^2.
       \end{equation}
Moreover, if $\gamma$ is a critical curve, there exist constants
%two real constants 
$u, v\in \R$, $u \neq 0$, such that
%there exists critical curve $\gamma$ arises as orbit of a 1-parameter subgroup of $G$, that is,
\[
\gamma(s) = A \exp{\left[X(u,v) s\right]} \cdot \lambda_0, 
%\quad g(s) =\exp{\left[Q(\k_1,\k_2) s\right]},
\]
where $s$ is the Lie arclength, $\lambda_0=[e_0,e_1]$ is the origin of $\Lambda$,
%\eqref{lie-sp}-null 2-space spanned by $e_0$, $e_2$, 
$A$ is a fixed element of {$G$}, and
%\[
%X(a_1,a_2)= \begin{pmatrix}
%a_1 & -a_2  & 0    & 1 & (a_1)^2-(a_2)^2 &  0\\
%a_2 & -a_1  & 1 & 0   &      0              &  -(a_1)^2+ (a_2)^2\\
%0              & 0                & 0    & 0   &      1              &  0\\
%0              &  0              &0  & 0   &      0              &  1\\
%1              &  0              & 0  & 0   &      a_1               & a_2\\
%0              &  -1              & 0  & 0   &      -a_2             &  -a_1\\
% \end{pmatrix} \in \mathfrak{g}
% \]
 \begin{equation}\label{X(a1,a2)}
 X(u,v)= 
 \begin{pmatrix} u I_{1,1} + v LI_{1,1} & L & (u^2-v^2)I_{1,1} \\ 
 0 & 0 & I_2\\ 
 I_{1,1} & 0 &  u I_{1,1} - v LI_{1,1} \end{pmatrix} \in \mathfrak{g},
 \end{equation}
 where $L=(\begin{smallmatrix} 0&1\\1&0\end{smallmatrix})$, $I_{1,1}=(\begin{smallmatrix} 1&0\\0&-1\end{smallmatrix})$, 
 and $I_2$ denotes the $2\times 2$ identity matrix.
\end{thmx}

\begin{proof}[Proof of Theorem \ref{thmB}]
The proof is divided in three parts.
%of Theorem \ref{thmA} 
%is organized in three steps.

%\subsection{The Pfaffian system of prolonged frames}\label{2.3}

\vskip0.2cm
\noindent \textbf{Part 1.} 
First, we show that the set of generic T-curves is in one-to-one correspondence 
with the set of integral curves of a suitable Pfaffian differential system.

\vskip0.1cm
Consider a generic T-curve $\gamma : J \to \Lambda$ parametrized by Lie arclength.
%Let $\gamma : J \to \Lambda$ be a generic T-curve parametrized by Lie arclength.
By Proposition \ref{Canonical-frame}, there exists
%the canonical frame of $\gamma$ defines 
a unique lift $\mathfrak{A} : J \to [G]$ of $\gamma$ to $[G]$. 
The map
\[
 \mathfrak{a} : J\ni s\longmapsto \big(\mathfrak{A}, \kappa_1(s), \kappa_2(s), \kappa_3(s), \kappa_4(s) \big)\in [G]\times\R^4,
 \]
where $\kappa_1 \neq 0$, is called the \textit{prolonged frame} of $\gamma$.
We call the product space $M := [G]\times\R^4$ 
%is called 
the \emph{configuration space} and 
denote
the coordinates on $\R^4$ 
%are denoted 
by  $(\kappa_1 , \kappa_2 , \kappa_3 , \kappa_4)$.

Slightly abusing notation,
%With some abuse of notation, 
we let $\omega^i_j$, $i,j\in\{0,\dots,5\}$, denote the entries of the Maurer--Cartan form of $[{G}]$,
as well as their pullbacks on 
%the configuration space 
$M$.
%
%By Proposition \ref{Canonical-frame}, the prolonged frames of $\gamma$ are the integral curves of 
On the configuration space $M$, we consider the Pfaffian
differential system $(\mathcal A, \eta)$ 
generated by the  
linearly independent 1-forms
%\begin{eqnarray*}
\begin{equation}\label{Pfaff-system}
\begin{array}{llll}
  \mu^1=\omega^2_0, &  \mu^2=\omega^2_1, & \mu^3 = \omega^3_0, & \mu^4 = \omega^3_1, \\
  \mu^5=\omega^2_4- \eta, &   \mu^6 = \omega^3_5 - \eta, &  \mu^7 = \omega^2_5, & \mu^8 = \omega^3_4, \\
  \mu^9 = \omega^0_0 - \kappa_1 \eta, &  \mu^{10} = \omega^1_1 + \kappa_1\eta, & \mu^{11} = \omega^1_0 -\kappa_2\eta, & \mu^{12} = \omega^0_1 + \kappa_2 \eta, \\
 \mu^{13} = \omega^3_2 - \kappa_3\eta, & \mu^{14} = \omega^0_4 - \kappa_4 \eta, &  & 
 \end{array}
 \end{equation}
%\end{eqnarray*}
subject to the \textit{independence condition} $$\eta =\omega^4_0 \neq 0.$$

\begin{remark}\label{r:linear-control-system}
Following \cite{GM},
%\cite[Definition 2.2, p. 306]{GM}, 
the Pfaffian system $(\mathcal{A},\eta)$ 
%defined by \eqref{Pfaff-system} 
is a \textit{linear control system} on 
%the Lie group 
$[G]$ associated to the 4-dimensional affine subspace $\mathbb{A} \subset \mathfrak{g}$ 
parametrized by 
\[
 \R^4 \ni\T(\k_1,\k_2,\k_3,\k_4) \mapsto 
  \begin{pmatrix} \k_1I_{1,1} + \k_2LI_{1,1} & L & \k_4I_{1,1} \\ 0 & \k_3L I_{1,1} & I_2\\ I_{1,1} & 0 &  \k_1I_{1,1} - \k_2LI_{1,1} \end{pmatrix}
  \in \mathfrak{g}.
\]
\end{remark}

By Proposition \ref{Canonical-frame}, the prolonged frames of $\gamma$ are integral curves of $(\mathcal A, \eta)$.
Conversely, if $\mathfrak{a} = \big( [A], \kappa_1 , \kappa_2 , \kappa_3 , \kappa_4\big) : J \to M$ is an integral curve 
of $(\mathcal{A}, \eta)$, then
$\gamma = [A_0,A_1]  : J \to \Lambda$ defines a generic T-curve, such that $[A]$ is its canonical frame, 
%field along $\gamma$,
and $\kappa_1 , \kappa_2 , \kappa_3 , \kappa_4$ its Lie curvatures. 
Accordingly, the integral curves of $(\mathcal{A},\eta)$ are the prolonged frames of generic T-curves 
in $\Lambda$.
%can be identified with the integral curves of the Pfaffian system $(\mathcal{A},\eta)$. 
%The latter will be referred to as the \textit{canonical (control) system}. 

Therefore, generic T-curves are in one-to-one correspondence with the integral curves of the Pfaffian system
$(\mathcal{A},\eta)$.

If we set
$$
 % \pi^1=\eta,\quad 
  \pi^1=d\kappa_1, \quad \pi^2 = d\kappa_2, \quad \pi^3 = d\kappa_3, \quad \pi^4 = d\kappa_4, 
  %\quad \pi^5 = d\kappa_4,
     $$
the 1-forms $\{\mu^a,\eta, \pi^\ell\}$, $a=1,\dots,14; \ell = 1, \dots,4$, define a global coframe on $M$.
Taking the exterior derivative and using the Maurer--Cartan equations \eqref{struct-eqs} imply the following structure equations 
for the coframe $\{\mu^a,\eta,\pi^\ell\}$:
\begin{equation}\label{quadratic1}
\begin{array}{l}
d\mu^1 \equiv \eta\wedge (\kappa_1\mu^1 + \kappa_2 \mu^2 + \kappa_3 \mu^3 + \mu^5), \\
d\mu^2 \equiv \eta\wedge (-\kappa_2\mu^1 - \kappa_1 \mu^2 + \kappa_3 \mu^3 - \mu^7),\\
d\mu^3 \equiv \eta\wedge (-\kappa_3\mu^1 + \kappa_1 \mu^3 + \kappa_2 \mu^4 + \mu^8),\\
d\mu^4 \equiv \eta\wedge (-\kappa_3\mu^2 - \kappa_2 \mu^3 - \kappa_1 \mu^4 - \mu^6),\\
d\mu^5 \equiv \eta\wedge (\kappa_4\mu^1 + \kappa_1 \mu^5 - \kappa_2 \mu^7 + \kappa_3\mu^8 + \mu^9 + 2 \mu^{10}),\\
d\mu^6 \equiv \eta\wedge (-\kappa_4\mu^4 - \kappa_1 \mu^6 - \kappa_3 \mu^7 + \kappa_2\mu^8 + 2\mu^9 +  \mu^{10}),\\
d\mu^7 \equiv \eta\wedge (-\kappa_4\mu^2 + \kappa_2 \mu^5 + \kappa_3 \mu^6 - \kappa_1\mu^7 +  \mu^{12} - \mu^{13}),\\
d\mu^8 \equiv \eta\wedge (\kappa_4\mu^3 - \kappa_3 \mu^5 - \kappa_2 \mu^6 + \kappa_1\mu^8 +  \mu^{11} + \mu^{13}),\\
d\mu^9 \equiv \eta\wedge (-\mu^3 + \kappa_1 \mu^9 + \kappa_1 \mu^{10} +\kappa_2\mu^{11}+ \kappa_2\mu^{12} +  \mu^{14} + \pi^{1}),\\
d\mu^{10} \equiv \eta\wedge (-\mu^2 - \kappa_1 \mu^9 - \kappa_1 \mu^{10} -\kappa_2\mu^{11}- \kappa_2\mu^{12} +  \mu^{14} - \pi^{1}),\\
d\mu^{11} \equiv \eta\wedge (-\mu^1 +2 \kappa_2 \mu^{10} +2\kappa_1\mu^{11}+ \pi^{2}),\\
d\mu^{12} \equiv \eta\wedge (-\mu^4 - 2 \kappa_2 \mu^{9} -2\kappa_1\mu^{12}- \pi^{2}),\\
d\mu^{13} \equiv \eta\wedge (-\mu^1 + \mu^{4} + \kappa_3\mu^{9}+ \kappa_3\mu^{10} {+ \pi^{3}}),\\
d\mu^{14} \equiv \eta\wedge (\mu^7 {- \mu^{8}} + 2\kappa_4\mu^{9}+ 2\kappa_4\mu^{10} + \pi^{4}),\\
\end{array}
\end{equation}
\begin{equation}\label{quadratic2}
\begin{array}{l}
d\eta \equiv - \eta \wedge (\mu^9 +\mu^{10}),\\
d\pi^1 = d\pi^1 = d\pi^3 = d\pi^4 = 0,
\end{array}
\end{equation}
where 
%the sign 
`$\equiv$' denotes congruence modulo the span of $\{ \mu^a\wedge\mu^b\}_{a,b = 1,\dots,14}$\,.

\begin{remark}\label{r:derived-flag}
It follows from \eqref{quadratic1} that the derived systems of $(\mathcal{A},\eta)$ are
\[
\mathcal{A}_1 = \mathrm{span}\{ \mu^1,\dots,\mu^8 \}, \quad \mathcal{A}_2 = \mathrm{span}\{ \mu^1,\dots,\mu^4 \}, \quad
\mathcal{A}_3 = \{0\},
\]
and hence the rank of each derived system is constant. (For the notion of derived system, see 
\cite{BCGGG,Gr,ILbook}.)
Alternatively, see \cite[Proposition 2.3]{GM}, where it is shown that all the derived systems of 
a linear control system $(\mathcal{A},\eta)$
%a linear control system 
%on a Lie group 
are of constant rank.
%%
%Thus, all the derived systems of $(\mathcal{A},\eta)$  have constant rank. For the notion of derived system, see 
%\cite{BCGGG,Gr}.
\end{remark}

\noindent \textbf{Part 2.}  
%
%We derive the Euler--Lagrange equations 
%by a construction (due to Griffiths \cite{Gr}) on an affine subbundle of $T^*(M)$ (cf. also \cite{Bryant1987,GM, Hsu}).
%
We now carry out a construction due to Griffiths \cite{Gr} on an affine subbundle of $T^*(M)$
 (see also \cite{Bryant1987,GM, Hsu}) in order to 
characterize the critical curves of $\mathcal{L}$.

%\vskip0.1cm
Let $\mathcal{Z}$ be the affine subbundle of $T^*(M)$ defined by
the 1-forms  $\{\mu^a\}$ and $\eta$, that is to say
$$
  \mathcal{Z}=\eta+\text{span}\{\mu^1,\dots,\mu^{14}\} \subset T^*(M).
     $$
The subbundle $\mathcal{Z}$ is called the \textit{phase space} of the Pfaffian system $(\mathcal{A},\eta)$.
The 1-forms  $(\mu^a, {\eta})$ induce a global affine trivialization of $\mathcal{Z}$, which in turn
may be identified with $M \times \mathbb{R}^{14}$ by the map
\[
%M\times \mathbb{R}^{14} \ni 
\left(([A], \kappa_1, \kappa_2,\kappa_3,\kappa_4), p_1, \dots, p_{14}\right) \mapsto 
%\eta_{|_{([\mathcal{F}], \kappa, \tau)}} 
{\eta_{|_{([A], \kappa_1, \kappa_2,\kappa_3,\kappa_4)}} }
+ \sum_{a=1}^{14} p_a{\mu^a}_{|_{([A], \kappa_1, \kappa_2,\kappa_3,\kappa_4)}},
%\in \mathcal{Z}\,,
\]
where $p_1,\dots, p_{14}$ are the fiber coordinates of the bundle map $\mathcal{Z} \to M$ with respect to the trivialization.
In this identification, the Liouville (canonical) 1-form of $T^*(M)$ restricted to $\mathcal{Z}$ takes the form
\begin{equation}\label{Liouville-form}
  \zeta= \eta  +  \sum_{a=1}^{14}p_a\mu^a\,. %+p_1\mu^1+\cdots +p_7\mu^7\,.
   \end{equation}
Taking the exterior derivative of \eqref{Liouville-form} and using of the quadratic equations \eqref{quadratic1} and \eqref{quadratic2} give
%\[
%  d\xi \equiv  d\eta  +  \sum_{a=1}^{14}dp_a\wedge \mu^a +\sum_{a=1}^{14}p_a d\mu^a
%   \]
%
\begin{equation*}
\begin{split}
  d\zeta & =    d\eta  +  \sum_{a=1}^{14}dp_a\wedge \mu^a +\sum_{a=1}^{14}p_a d\mu^a \equiv \\
 % - \eta \wedge (\mu^9 +\mu^{10})  %inseriti nei termini \mu^9 e \mu^10
 &[dp_1 + (p_1\k_1 -p_2\k_2 -p_3\k_3 +p_5 \k_4 -p_{11} - p_{13})\eta]\wedge \mu^1 +\\
 &[dp_2 + (p_1\k_2 -p_2\k_1 -p_4\k_3 -p_7 \k_4 -p_{10})\eta]\wedge \mu^2 +\\
 &[dp_3 + (p_1\k_3 +p_3\k_1 -p_4\k_2 +p_8 \k_4 -p_{9})\eta]\wedge \mu^3 +\\
 &[dp_4 + (p_2\k_3 +p_3\k_2 -p_4\k_1 -p_6 \k_4 -p_{12}+p_{13})\eta]\wedge \mu^4 +\\
 &[dp_5 + (p_1 +p_5\k_1 +p_7\k_2 -p_8\k_3 )\eta]\wedge \mu^5 +\\
 &[dp_6 + (-p_4 - p_6\k_1 +p_7\k_3 -p_8\k_2)\eta]\wedge \mu^6 +\\
 &[dp_7 + (-p_2  -p_5\k_2 -p_6\k_3 -p_7 \k_1 + p_{14})\eta]\wedge \mu^7 +\\
 &[dp_8 + (p_3 +p_5\k_3 +p_6\k_2 +p_8 \k_1 {-p_{14}})\eta]\wedge \mu^8 +\\
 &[dp_9 + ({-1} +p_5 +2p_6 +p_9\k_1 -p_{10} \k_1 -2p_{12}\k_2 +p_{13}\k_3 + 2p_{14}\k_4)\eta]\wedge \mu^9 +\\
 &[dp_{10} + ({-1} + 2p_5 +p_6 +p_9\k_1 -p_{10} \k_1 +2p_{11}\k_2 +p_{13}\k_3 + 2p_{14}\k_4)\eta]\wedge \mu^{10} +\\
 &[dp_{11} + (p_8 +p_9\k_2 -p_{10} \k_2 +2p_{11}\k_1)\eta]\wedge \mu^{11} +\\
 &[dp_{12} + (p_7 +p_9\k_2 -p_{10} \k_2 -2p_{12}\k_1)\eta]\wedge \mu^{12} +\\
 &[dp_{13} + (-p_7 +p_8)\eta]\wedge \mu^{13} + [dp_{14} + (p_9 +p_{10})\eta]\wedge \mu^{14} +\\
 &(p_9 -p_{10})\eta\wedge  \pi^1 + (p_{11} -p_{12})\eta\wedge  \pi^2 {+ p_{13}} \eta \wedge\pi^3 + p_{14}\eta\wedge \pi^4,
  \end{split}
\end{equation*}
where `$\equiv$' denotes congruence modulo the span of $\{\mu^a\wedge\mu^b\}_{a,b = 1,\dots,14}$.
%
%\vskip0.1cm

Let $(\mathcal{C}(d\zeta), \eta)$ be the Cartan system of the
%$\mathcal{C}(d\zeta)\subset T^*(\mathcal{Z})$ determined by the
2-form $d\zeta$, i.e., the Pfaffian system on $\mathcal{Z}$ generated by the 1-forms
$\left\{X {\lrcorner}  \,d\zeta \mid X \in \mathfrak{X}(\mathcal{Z}) \right\}\subset \Omega^1(\mathcal{Z})$,
%\[
% \left\{X {\lrcorner}  \,d\zeta \mid X \in \mathfrak{X}(\mathcal{Z}) \right\}\subset \Omega^1(\mathcal{Z}),
%  \]
with independence condition $\eta \neq 0$.

%\vskip0.1cm

By Part 1, we know that generic T-curves are in one-to-one correspondence with the integral curves of the 
Pfaffian system $(\mathcal{A},\eta)$. 

Now, let $\mathfrak{a} : J \to M$ be the prolonged frame associated to a generic T-curve 
$\gamma : J \to \Lambda$ parametrized by Lie arclength.
By Griffiths' approach to the calculus of variations (see, for instance, \cite{Bryant1987, GM, Gr, Hsu}),
if $\mathfrak{a}$ 
%of the generic transversal curve $\gamma : J \to \S$
%parametrized by the natural parameter, 
admits a lift $\mathfrak{z}: J \to \mathcal{Z}$ to the phase space
$\mathcal{Z}$ which is an integral curve of the Cartan system $(\mathcal{C}(d\zeta), \eta)$,
it follows that $\gamma$ is a critical curve of $\mathcal{L}$, with respect to compactly supported variations.

%As observed by
%According to
In \cite[p. 70]{Bryant1987} (see also \cite{Hsu}), it is observed that if the derived systems 
of $(\mathcal{A},\eta)$ have all 
constant rank, as in the case under consideration
%the case under discussion 
(see Remark \ref{r:derived-flag}), then the converse is also true,
%then the converse is also true, 
that is,
all critical curves are projections of integral curves of the Cartan 
system $(\mathcal{C}(d\zeta), \eta)$.
%that is,
If $\mathfrak{z} : J \to \mathcal{Z}$ is a curve in the phase space which is an integral curve
of the Cartan system $(\mathcal{C}(d\zeta), \eta)$  and if $\mathtt{pr} : \mathcal{Z} \to \Lambda$
is the natural projection of $\mathcal{Z}$ onto $\Lambda$, then $\gamma= \mathtt{pr}\circ \mathfrak{z} : J \to \Lambda$
is a critical curve of $\mathcal{L}$ with respect to compactly supported variations. All critical curves 
%of $\mathcal{W}$ 
arise in this way.

\vskip0.1cm
To compute the Cartan system $(\mathcal{C}(d\zeta), \eta)$,
we take $(\mu^a,\eta,\pi^\ell, dp_a)$ as a coframe on $\mathcal{Z}$
%and compute the Cartan system $(\mathcal{C}(d\zeta), \eta)$ by
%
and contract $d\zeta$ with $\{{\partial}/{\partial{ p_a}}, {\partial}/{\partial{ \mu^a}},
   {\partial}/{\partial{ \pi^\ell}},{\partial}/{\partial\eta}\}$, respectively.
%$$
% \Big(\frac{\partial}{\partial{ \mu^a}},\frac{\partial}{\partial\eta},
%   \frac{\partial}{\partial{ \pi^\ell}},\frac{\partial}{\partial{ p_a}}\Big)
%    $$
%on $\mathcal{Z}$, dual to the coframe $(\mu^a,\eta,\pi^\ell,dp_a)$,
%
%the integral curves of the Cartan system $(\mathcal{C}(d\zeta), \eta)$ must satisfy the following equations:
Thus the Cartan system on $\mathcal{Z}$ is given by the Pfaffian equations

\begin{equation}\label{cartan-sys-eqs1}
 \mu^a =0, \quad a = 1, \dots, 14,
\end{equation}
\begin{eqnarray}
%\begin{split}
%  d\zeta & =    d\eta  +  \sum_{a=1}^{14}dp_a\wedge \mu^a +\sum_{a=1}^{14}p_a d\mu^a \equiv \\
 % - \eta \wedge (\mu^9 +\mu^{10})  %inseriti nei termini \mu^9 e \mu^10
 &&dp_1 + (p_1\k_1 -p_2\k_2 -p_3\k_3 +p_5 \k_4 -p_{11} - p_{13})\eta=0 \label{cs1},\\
 &&dp_2 + (p_1\k_2 -p_2\k_1 -p_4\k_3 -p_7 \k_4 -p_{10})\eta=0\label{cs2},\\
 &&dp_3 + (p_1\k_3 +p_3\k_1 -p_4\k_2 +p_8 \k_4 -p_{9})\eta=0\label{cs3},\\
 &&dp_4 + (p_2\k_3 +p_3\k_2 -p_4\k_1 -p_6 \k_4 -p_{12}+p_{13})\eta=0\label{cs4},\\
 &&dp_5 + (p_1 +p_5\k_1 +p_7\k_2 -p_8\k_3 )\eta=0\label{cs5},\\
& &dp_6 + (-p_4 - p_6\k_1 +p_7\k_3 -p_8\k_2)\eta=0\label{cs6},\\
 &&dp_7 + (-p_2  -p_5\k_2 -p_6\k_3 -p_7 \k_1 + p_{14})\eta=0\label{cs7},\\
 &&dp_8 + (p_3 +p_5\k_3 +p_6\k_2 +p_8 \k_1 {-p_{14}})\eta=0\label{cs8},\\
 &&dp_9 + ({-1} +p_5 +2p_6 +p_9\k_1 -p_{10} \k_1 -2p_{12}\k_2 +p_{13}\k_3 + 2p_{14}\k_4)\eta=0\label{cs9},\\
& &dp_{10} + ({-1} + 2p_5 +p_6 +p_9\k_1 -p_{10} \k_1 +2p_{11}\k_2 +p_{13}\k_3 + 2p_{14}\k_4)\eta=0\label{cs10},\\
& &dp_{11} + (p_8 +p_9\k_2 -p_{10} \k_2 +2p_{11}\k_1)\eta=0\label{cs11},\\
& &dp_{12} + (p_7 +p_9\k_2 -p_{10} \k_2 -2p_{12}\k_1)\eta=0\label{cs12},\\
%& & p_{9}=p_{10}=p_{13}=p_{14}=0, \quad p_7 = p_8, \quad p_{11}=p_{12}. \label{cs13}
%\\
%\end{eqnarray}
%\begin{eqnarray}
%%%% in alternativa
& &dp_{13} + (-p_7 +p_8)\eta=0 \label{cs13},\\
& &dp_{14} + (p_9 +p_{10})\eta=0\label{cs14},\\
& &(p_9 -p_{10})\eta =0    \label{cs15},\\ %\wedge  \pi^1 
&& (p_{11} -p_{12})\eta=0 \label{cs16},\\    %\wedge  \pi^2 
&& {p_{13}} \eta =0 \label{cs17},\\% \wedge\pi^3 + 
&& p_{14}\eta = 0  \label{cs18}, \\%\wedge \pi^4,
&&(p_9 -p_{10})  \pi^1 + (p_{11} -p_{12})  \pi^2 {+ p_{13}} \pi^3 + p_{14}  \pi^4 =0 \label{cs19},
%&&\eta \neq 0
%  \end{split}
\end{eqnarray}
subject to the independence condition %$\eta \neq 0$.
\begin{equation}
\eta \neq 0.
\end{equation}

%\vskip0.1cm
We are now in a position to prove that a generic T-curve $\gamma$ is critical 
%for $\mathcal{W}$ 
if and only if the Lie curvatures of $\gamma$ are constant and satisfy \eqref{conds}. 
%The proof of the first statement of Theorem \ref{thmA} 
This amounts to proving that
%by proving that
%Next we prove that 
a curve $\mathfrak{z} : J \to \mathcal{Z}$ is an integral curve of the Cartan system $(\mathcal{C}(d\zeta), \eta)$
if and only if the Lie curvatures of the projection $\gamma= \mathtt{pr}\circ \mathfrak{z}$ are constant and satisfy \eqref{conds}.

%the conditions \eqref{conds}. 
Suppose $\mathfrak{z} : J \to \mathcal{Z}$,
\[
\mathfrak{z}(s) = \left([A(s)], \k_1(s), \k_2(s), \k_3(s), \k_4(s), p_1(s), \dots, p_{14}(s)\right),
\]
is an integral curve of the Cartan system $(\mathcal{C}(d\zeta), \eta)$, 
%with $\eta = ds$,
where $A: J \to G$ is a smooth map, $\k_1, \k_2, \k_3, \k_4, p_1, \dots, p_{14}$ are smooth functions,
and $ds = \eta$.
%is an integral curve of the Cartan system $(\mathcal{C}(d\zeta), \eta)$, with $\eta = ds$. 
From \eqref{cartan-sys-eqs1} it follows that
\[
\mathfrak{z}^*(\mu^a) = 0 \quad (a = 1, \dots, 14), \quad \mathfrak{z}^*(\eta) \neq 0,
  \]
which implies that $\left([A], \k_1, \k_2, \k_3, \k_4\right)$ is an integral curve of the Pfaffian system $(\mathcal{A}, \eta)$.
Hence the projection $\gamma= \mathtt{pr}\circ \mathfrak{z}$,
\[
 \gamma(s)= [A_0(s), A_1(s)],
  \]
  is a generic T-curve in $\Lambda$, with Lie curvatures $\k_1, \k_2, \k_3, \k_4$, 
  and $A$ is a fourth order Lie frame field along $\gamma$.
Next, 
%subject to the independence condition $\mathfrak{z}(\eta)\neq 0$, 
equations \eqref{cs13}--\eqref{cs19}, together with the independence condition $\mathfrak{z}^*(\eta)\neq 0$, give
\begin{equation}\label{consq-cs1}
  p_{9}=p_{10}=p_{13}=p_{14}=0, \quad p_7 = p_8, \quad p_{11}=p_{12}. 
   \end{equation}
Substitution of \eqref{consq-cs1} into \eqref{cs11} and \eqref{cs12} together with the condition $\kappa_1\neq 0$ gives
\begin{equation}\label{consq-cs2} %\label{cs14}
  p_{11} =p_{12}= p_7 = p_8 = 0.
   \end{equation}
Substitution of \eqref{consq-cs1} and \eqref{consq-cs2} into\eqref{cs9} and \eqref{cs10} yields
\begin{equation}\label{consq-cs3}  %\label{cs15}
  p_5 = p_6 = \frac13.
   \end{equation}
From \eqref{cs7} and \eqref{cs8}, we then obtain
\begin{equation}\label{consq-cs3.1} %\label{cs16}
  p_2 = -\frac13(\k_2 +\k_3), \quad  p_3 = -\frac13(\k_2 +\k_3).
   \end{equation}
Now, \eqref{cs5} and \eqref{cs6} imply
 \begin{equation}\label{consq-cs4} %\label{cs16}
    p_1 = p_4 = -\frac13\k_1.
    \end{equation}
From \eqref{cs1} and \eqref{cs4}, we then obtain
\begin{equation}\label{consq-cs5}
  \k_4  = \k_1^2 - (\k_2 + \k_3)^2, \quad \k_1= \mathrm{constant} \neq 0.
    \end{equation}
Finally, \eqref{cs2} and \eqref{cs3} together with the condition $\k_1\neq 0$ give
\[
  \k_3 =0, \quad \k_2 = \mathrm{constant}.
   \]
%which implies that $\k_2$ is a constant function, 
Thus the Lie curvatures of $\gamma$ are constant and satisfy \eqref{conds}, as claimed.

\vskip0.1cm
Conversely, let $\gamma :J \to \Lambda$ be a generic T-curve whose Lie curvatures
are constant and satisfy \eqref{conds} and let $A$ be a fourth order Lie frame field along $\gamma$. Then
\[
\mathfrak{z}(s) = 
\Big([A(s)], \k_1, \k_2, 0, \k_1^2 -\k_2^2, 
-\frac{\k_1}{3},-\frac{\k_2}{3} ,-\frac{\k_2}{3} , -\frac{\k_1}{3}, \frac13, \frac13,0,0,0,0,0,0,0,0\Big)
\]
is a lift of $\gamma$ to the phase space $\mathcal{Z}$. By construction, $\mathfrak{z}$
%such that $\mathfrak{z}^*(\eta)\neq 0$ and by construction $\mathfrak{z}$ 
is an integral curve of the Cartan system $(\mathcal{C}(d\zeta), \eta)$.

\vskip0.2cm
%\noindent \textbf{Step 3.}  
\noindent \textbf{Part 3.}  
Finally, we provide 
%prove the second statement of Theorem \ref{thmA} about
the explicit integration of the critical curves of $\mathcal{L}$.

\vskip0.1cm
Let $\gamma$ be a critical curve of $\mathcal{L}$, let $\mathfrak{A}$ be the canonical
frame field along $\gamma$, and let $s$ be the Lie arclength of $\gamma$. From the first part of the proof, we have
\begin{equation}\label{eq-for-CF}
\mathfrak{A}^{-1} d\mathfrak{A} = X(u,v)ds, 
\end{equation}
where $X(u,v)$ is given by \eqref{X(a1,a2)} and $u$, $v\in \R$ satisfy the conditions \eqref{conds}.
It follows from \eqref{eq-for-CF} that $\mathfrak{A}(s) = A \exp [X(u,v) s]$, for some $A\in G$, and hence
$\gamma$ is determined by the 1-parameter subgroup $\exp [X(u,v) s]$, up to the left multiplication by an element $A\in G$.

Conversely, if $\gamma(s) = A \exp [X(u,v) s] \cdot [e_0,e_1]$, then $\mathfrak{A}(s) = [A \exp [X(u,v) s]]$
defines a canonical frame field along $\gamma$. By the discussion in the first part of the proof, this 
implies that $\gamma$ is a critical curve of $\mathcal{L}$.
\end{proof}

\begin{remark}
The Cartan system $(\mathcal{C}(d\zeta), \eta)$ 
can be reduced, in the sense that
%is reducible, meaning that
%According to Griffiths \cite{GM,Gr},
%the
there is a 
%According to Griffiths \cite{GM,Gr}, the 
nonempty submanifold
$\mathcal{Y} \subseteq \mathcal{Z}$,
%called the reduced space, 
such that: (1) at each point of $\mathcal{Y}$ 
there exists an integral element
of $(\mathcal{C}(d\zeta), \eta)$ tangent to $\mathcal{Y}$; and (2) if $\mathcal{X} \subseteq \mathcal{Z}$ is any other
submanifold with the same property of $\mathcal{Y}$, then $\mathcal{X} \subseteq\mathcal{Y}$.
Following Griffiths \cite{GM,Gr}, the submanifold $\mathcal{Y}$ 
is called the \textit{momentum space} 
%of the variational problem.
%
%In addition, 
and the restriction $(\mathcal{J}, \eta)$ to $\mathcal{Y}$ of the Cartan system $(\mathcal{C}(d\zeta), \eta)$ 
is the \textit{Euler--Lagrange system} of the variational problem. 
By the above discussion, %the momentum space 
$\mathcal{Y}$ 
is the $19$-dimensional submanifold of $\mathcal{Z} \cong M\times \R^{14}$,
$M = [G]\times \R^4$, defined by the
following set of equations
\begin{eqnarray*}
&&p_7 = p_8= p_{9}=p_{10}= p_{11}=p_{12}=p_{13}=p_{14}=0, \\
&&p_1 = p_4 = -\frac13\k_1,
\quad p_2 = p_3 = -\frac13(\k_2+\k_3),
\quad p_5 = p_6 = \frac13,
\end{eqnarray*}
where, as before, $\k_1,\dots,\k_4$ are the coordinates of $\R^4$ and $p_1,\dots,p_{14}$ are the coordinates of $\R^{14}$.
The Euler--Lagrange system $(\mathcal{J}, \eta)$ 
%on the momentum space $\mathcal{Y}$ 
is given 
%on $\mathcal{Y}$ 
by the Pfaffian equations
\begin{eqnarray*}
&& \mu^a =0, \quad a = 1, \dots, 14,\\
&&d\k_1 - \left( \k_4 -\k_1^2 +(\k_2 +\k_3)^2 \right)\eta =0, \\
&&d(\k_2 +\k_3) = 0, \quad d(\k_2 +\k_3)  + 2\k_1\k_3 \eta =0, \\
 &&d\k_1 + \left( \k_4 -\k_1^2 +(\k_2 +\k_3)^2  \right)\eta =0,
\end{eqnarray*}
with the independence condition $\eta\neq 0$.
%
%A basic result states that 
It is known that
the Pfaffian systems $(\mathcal{C}(d\zeta), \eta)$ and $(\mathcal{J}, \eta)$ have
the same integral curves \cite{GM,Gr}.
%
%The momentum space $\mathcal{Y}$ is the 
%

\end{remark}

\subsection{Open questions and further developments}
{A critical curve 
%(which has constant curvature) 
is congruent to
%$\gamma(s) =\exp[sX(u,v)]\cdot \lambda_0$, 
the orbit through $\lambda_0 =[e_0,e_1]$ of the 1-parameter group of
Lie sphere transformations $\mathbb{R} \ni s \mapsto \exp(sX) \in G$,
where $X(u,v)$ is a fixed element of $\mathfrak{g}$ (cf. \eqref{X(a1,a2)}). 
Thus, the explicit determination of the critical curves reduces to the computation
of the exponentials of the matrices $s X(u,v)$.
From a computational point of view, this
requires a detailed analysis of the possible orbit types of the infinitesimal generator $X(u,v)$.
In this respect, we recall that
a closed, finite formula for the exponential map
from the Lie algebra to the identity component of
the pseudo-orthogonal group $\mathrm{SO}(2,4)$ is given in \cite{Barut}.
However, 
the parametrizations obtained in this way lack in general a direct geometric interpretation
suitable for the description of critical curve. 
This involves some nontrivial additional work.
Another interesting question to address is that of the existence and determination of closed critical curves. 
We observe that the trajectory of a critical curve can be closed if and only if the characteristic 
polynomial of the infinitesimal generators $X(u,v)$ has purely imaginary roots.
%
%comments on the eigenvalues
It is our intention to address the above questions in future work.}

\bibliographystyle{amsalpha}

\end{document}